\documentclass{amsart}
\usepackage[utf8x]{inputenc}
\usepackage{amsmath,amsthm}
\usepackage{amsfonts,amstext,amssymb, mathtools, comment, mathtools}
\usepackage{graphicx}

\RequirePackage[OT1]{fontenc}
\RequirePackage{amsthm,amsmath}
\RequirePackage[numbers]{natbib}
\RequirePackage[colorlinks,citecolor=blue,urlcolor=blue]{hyperref}

\newcommand{\levy}{L\'{e}vy }
\newcommand{\p}{{\mathbb P}}
\newcommand{\e}{{\mathbb E}}
\newcommand{\D}{{\mathrm d}}
\newcommand{\uX}{{\underline X}}
\newcommand{\oX}{{\overline X}}
\newcommand{\N}{{\mathbb N}}
\newcommand{\R}{{\mathbb R}}
\newcommand{\1}[1]{\mbox{\rm\large  1}_{\{#1\}}}

\renewcommand{\a}{{\alpha}}
\newcommand{\om}{{\omega}}

\renewcommand{\i}{{\mathrm i}}

\newtheorem{theorem}{Theorem}
\newtheorem{corollary}{Corollary}
\newtheorem{lemma}{Lemma}
\newtheorem{prop}{Proposition}
\newtheorem{remark}{Remark}

\begin{document}

\title[A process on the real line seen from its supremum]{A L\'evy process on the real line seen from its supremum and max-stable processes}

\author[S.\ Engelke]{Sebastian Engelke}

\author[J.\ Ivanovs]{Jevgenijs Ivanovs}

%
%
%
%
%

\begin{abstract}
We consider a process $Z$ on the real line composed from a \levy process and its exponentially tilted version killed with arbitrary rates
and give an expression for the joint law of $Z$ seen from its supremum, the supremum $\overline Z$
and the time $T$ at which the supremum occurs. In fact, it is closely related 
to the laws of the original and the tilted \levy processes conditioned to stay negative and positive. 
The result is used to derive a new representation of stationary particle systems
driven by L\'evy processes. In particular, this 
implies that a max-stable process arising from \levy processes admits a mixed moving maxima representation with spectral functions given by the conditioned \levy processes.
\end{abstract}


\maketitle

\section{Introduction}
Let $X=(X(t))_{t\geq 0}$ be a general \levy process, but not a compound Poisson process, and assume that $X$ drifts to $-\infty$ as $t\to\infty$.
It is well-known that such a process splits at its unique supremum into two independent parts,
where the post-supremum process has the law of $X$ conditioned to stay negative and the defective pre-supremum process (look backwards and down from the supremum) has the law of $X$ conditioned to stay positive,
see~\cite{bertoin_splitting, chaumont_french, duquesne, chaumont_doney}. 
We note that when $X$ drifts to $-\infty$ the term `conditioned to stay positive' has certain ambiguity~\cite{hirano_conditioning}, and so we avoid using it in this case in the following.
It turns out that a similar representation holds true if the process $X$ is
suitably extended to the real line. This leads to an important application to \levy driven particle systems.

Consider the Laplace exponent 
$\psi(\theta)=\log\e e^{\theta X(1)}$ and assume that $\psi(\nu)=0$
for some $\nu > 0$.
Let $X^\nu$ be an independent
\levy process with Laplace exponent $\psi(\theta+\nu)$ called the \emph{associated} or \emph{exponentially tilted} process.  
It is well-known that $X$ drifts to $-\infty$ and $X^\nu$ drifts
to $+\infty$ for $t\to\infty$. Define the c\`adl\`ag process $Z$ 
on the real line by
\begin{align}\label{proc_Z}
  Z(t)=\1{t\geq 0}X(t)-\1{t<0}X^\nu((-t)-), \quad t\in\mathbb R,
\end{align}
and denote by $\overline Z$ the supremum of the process $Z$ and by $T$
the time at which the supremum occurs. In this paper we give an expression
for the joint law of the process $Z$ shifted with its supremum
point into the origin, together with the supremum point, that is we 
specify the measure
\begin{align}\label{joint_law}
  \p( (Z(T + s) - \overline Z)_{s\in\R} \in B, T\in \D t, \overline
Z \in \D x).
\end{align}
In fact, this law is
closely related to the law of a process $Y$ obtained as the process $Z$
`conditioned to stay negative', see~\eqref{eq:proc_Y}. The problem of multiple possible definitions of a conditioned process does not arise in our case, 
because $Z(t)\rightarrow -\infty$ as $|t|\rightarrow \infty$.
This result holds in a more general framework where we  
only assume that for some $\nu\in\R$ the Laplace exponent $\psi(\nu)$
is finite and the L\'evy processes $X$ and $X^\nu$ are killed
with arbitrary exponential rates. 

In the case $\psi(\nu)=0$, the process $-X^\nu$ can be seen as the 
process $X$ reversed in time with respect to its invariant measure
$\pi(\D x) = e^{-\nu x} \D x$, since for any $t> 0$ and Borel subsets
$C,D\subset \R$ we have
\begin{align*}
  \int_C \p( x + X(t) \in D) \pi(\D x) = \int_D \p( y - X^\nu(t) \in C) \pi(\D y).
\end{align*}
Let $\{U_i, i\in\N\}$ be a Poisson point process (PPP) on $\R$
with intensity measure $\pi(\D x)$ and let $Z_i$, $i\in\N$, be independent copies of
the process $Z$. The above implies that the Poisson point process 
\begin{align*}
  \Psi_1 = \{U_i + Z_i, i\in\N\}
\end{align*}
of particles started at the $U_i$'s and moving along the trajectories
of $X$ for $t \geq 0$ and $-X^\nu$ for $t<0$, respectively, is stationary, see also Section~\ref{sec:particle_systems} and references therein.

From the perspective of extreme value theory, the process $\eta$ of pointwise maxima of the 
system $\Psi_1$
\begin{equation}\label{eq:maxima}
  \eta(t) = \max_{i\in\N} U_i + Z_i(t), \quad t\in\R,
\end{equation}
is well-known. It follows from \cite{bro1970,sto2008,eng2013a} that $\eta$ is stationary,
has c\`adl\`ag paths and is max-stable.
The latter means that for any $n\in\N$
and independent copies $\eta_1,\dots, \eta_n$ of $\eta$, the process
$\max_{i=1, \dots, n} \eta_i -\log n$ has the same distribution as $\eta$
(cf., \cite{deh2006a}). 
For instance, if $B(t),t\geq 0$ is a standard Brownian motion
and $X(t) = B(t) - t/2$, $t\geq 0$,
then $Z(t) = B(t) - |t|/2$, $t\in\R$, and $\eta$ coincides with the original
definition of the Brown-Resnick process in \cite{bro1977}.
Its extension to Gaussian random fields in \cite{kab2009} has become
a standard model in extreme value statistics for assessing the risk
of rare meteorological events. 

It was asked in \cite{sto2008} whether for a general L\'evy processes $X$
the max-stable process~$\eta$ 
possesses a stochastic representation as a mixed moving maxima process
\begin{align}\label{MMM}
  \max_{i\in\N} V_i + F_i(t - T_i), \quad t\in\R,
\end{align}
for some Poisson point process $\{(F_i,T_i,V_i), i\in\N\}$ on $\mathcal D\times \R\times \R$ with intensity measure
$C_0 \, \p_F(\D \omega) \, \D t \, e^{-\nu v}\D v $, where $C_0 > 0$ is a constant. Here $\mathcal D$ is the space of c\`adl\`ag functions on the real line, and 
$\p_F$ is the law of a stochastic process on the real line called
the spectral process.
The existence of such a representation
is important as it implies that the process is mixing and can be 
efficiently simulated. For the original Brown-Resnick process, that is $Z(t) = B(t) - |t|/2$, the answer is positive. Indeed, \cite{kab2009} prove the existence
and \cite{engelke} show that the spectral functions
in this case are given by $3$-dimensional (drifted) Bessel processes.

Applying the new expression for the joint law of \eqref{joint_law} given in Corollary~\ref{thm:cor},
we show that for a general \levy processes $X$ there is
a stochastic representation of $\eta$ as a mixed moving maxima process.
More importantly, we derive the explicit distribution $\p_F$ of the spectral
processes $F_i$ in \eqref{MMM}. It turns out that $\p_F$ is the law of $Y$, that
is, the process $Z$ conditioned to stay negative. 
We envisage that Theorem~\ref{thm:main} will find a similar application to a more general particle system, where particles can die and be born. 

In Section \ref{sec:main} we give necessary preliminaries and
state the two main theorems, that is, the identity relating 
the law of \eqref{joint_law} with $Y$, and the mixed moving maxima
representation of $\eta$. The proof of the former is postponed
to Section \ref{sec:proofs} where we use It\^o's excursion theory and the recent result from~\cite{chaumont_supremum}
to analyze the process $Z$ seen from its supremum. 
As a side result we relate the excursion measures of the tilted process to the ones of the original process in Proposition~\ref{prop:tilting}.
Finally, Section~\ref{sec:conditioned} discusses possible approaches to simulation of the process $\eta$
based on its mixed moving maxima representation.

\section{Main results}\label{sec:main}
\subsection{Two L\'evy processes}\label{sec:levy}
Let us first fix some notation. 
Let $(\Omega,\mathcal F,\p)$ be a probability space equipped with a filtration $(\mathcal F_t)_{t\geq 0}$, satisfying the usual conditions.
Let also $X=(X(t))_{t\geq 0}$, be a 
L\'evy process on this filtered probability space with characteristic triplet $(a,\sigma,\Pi)$, that is
\begin{equation}\label{eq:psi}
 \psi(\theta)=\log \e e^{\theta X(1)}=a\theta+\frac{1}{2}\sigma^2\theta^2+\int_\R (e^{\theta x}-1-\theta x\1{|x|<1})\Pi(\D x),
\end{equation}
where $\sigma^2\geq 0$ and $\Pi$ are the variance
of the Brownian component and the L\'evy measure, respectively. 
The so-called Laplace exponent $\psi(\theta)$ is finite for $\theta\in\i\R$, but may be infinite for some $\theta\in\R$.
For details on L\'evy processes we refer the reader to~\cite{bertoin,kyprianou}.
Throughout this work we assume that $X$ is not a process with monotone paths, neither it is a Compound Poisson Process (CPP), but see also Remark~\ref{rem:CPP}.

Pick $\nu\in\R$ such that $\psi(\nu)<\infty$ which is equivalent to $\int_{|x|>1}e^{\nu x}\Pi(\D x)<\infty$ according to~\cite[Thm.\ 3.6]{kyprianou}.
One can see that $\psi(\theta)<\infty$ for all $\theta\in[0,\nu]$ if $\nu>0$ and $\theta\in[\nu,0]$ if $\nu< 0$.
Moreover, one can define an exponentially tilted measure with respect to $\nu$, also known as the Esscher transform:
\[\left.\frac{\D \p^\nu}{\D \p}\right|_{\mathcal F_t}=e^{\nu X(t)-\psi(\nu)t},t\geq 0.\]
It is known that $X$ under $\p^\nu$ is a \levy process, say $X^\nu$, with Laplace exponent $\psi^\nu(\theta)=\psi(\theta+\nu)-\psi(\nu)$,
which implies that $\sigma^\nu=\sigma$ and $\Pi^\nu(\D x)=e^{\nu x}\Pi(\D x)$, see e.g.~\cite{kyprianou}. 
Furthermore, $X$ has paths of bounded variation on compacts if and only if so does $X^\nu$, 
in which case~\eqref{eq:psi} can be written as
\[\psi(\theta)=\hat a\theta+\int_{\mathbb R}(e^{\theta x}-1)\Pi(\D x),\]
where $\hat a \in \R$ is the linear drift, and then $\hat a^\nu=\hat a$.
This furthermore shows that $X^\nu$ is not a process with monotone paths either, and neither it is a CPP.

The case $\nu>0$, $\psi(\nu)=0$, will be of special interest. In this case $\e X(1)<0$ and $\e^\nu X(1)>0$, which follows from the convexity of $\psi(\theta)$ on $[0,\nu]$, see e.g.~\cite[Ch.~3]{kyprianou}.
This implies that $X$ drifts to $-\infty$ and $X^\nu$ drifts to $+\infty$.

In addition, we will allow for defective (or killed) processes. 
We say that $X$ and $X^\nu$ are killed at rates $q>0$ and $p>0$ if they are sent to an additional `cemetery' state $\partial$ at the times $e_q$ and $e_p$ respectively,
where $e_q$ denotes an exponentially distributed random variable of rate $q$ independent of everything else. We let $\zeta$ and $\zeta^\nu$ be the life times of $X$ and $X^\nu$ respectively.

\subsection{Two processes on the real line}\label{sec:processes}
Consider two independent \levy processes $X$ and $X^\nu$ killed at rates $q>0$ and $p>0$ (with respective life times $\zeta$ and $\zeta^\nu$) as defined in Section~\ref{sec:levy}.
Define a c\`adl\`ag process $Z$ on the real line:
\[Z(t)=\1{t\geq 0}X(t)-\1{t<0}X^\nu((-t)-)\]
for $t\in [-\zeta^\nu,\zeta)$ and put $Z(t)=\partial$ otherwise. The left hand side of Figure~\ref{fig:ZY} illustrates
the construction of $Z$. Roughly speaking, the process $Z(t),t\leq 0$ seen with respect to `small' axis is $X^\nu$, which may help to better understand various relations in the following.
\begin{figure}[h]
\caption{Schematic sample paths of $Z$ and $Y$.}\label{fig:ZY} 
\begin{center}
\includegraphics{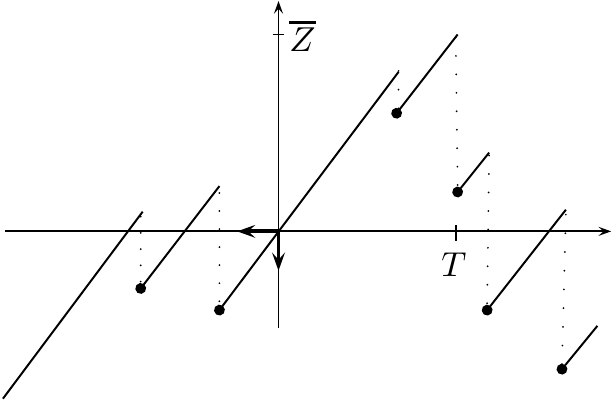}\hspace{0.3in}
\includegraphics{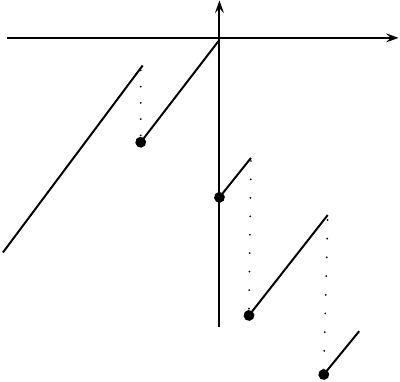}
\end{center}
\end{figure}
We remark that for $0\leq s\leq t$, given that $\zeta^\nu>t$, it holds that $Z(-t+s)-Z(-t)$ has the same distribution as $X^\nu(s)$.
Furthermore, if $X$ has no positive (negative) jumps then $Z$ has no positive (negative) jumps either.
For simplicity of notation we assume that $\partial\vee x=x$ and $\partial \wedge x=x$ for any $x\in\R$. Define the overall supremum and its time
\begin{align}\label{eq:Z}
 &\overline Z=\sup_{t\in[-\zeta^\nu,\zeta)}\{Z(t)\}, &T=\inf\{t\in\R: Z(t)\vee Z(t-)=\overline Z\}.
\end{align}
It turns out that the law of the process $Z$ can be described by another process $Y$ which we now define. Letting
\begin{align*}
&\overline X=\sup_{t\in[0,\zeta)}\{X(t)\},&\overline T=\inf\{t\geq 0: X(t)\vee X(t-)=\overline X\}, \\
&\underline X^\nu=\inf_{t\in[0,\zeta^\nu)}\{X^\nu(t)\},&\underline T^\nu=\inf\{t\geq: X^\nu(t)\wedge X^\nu(t-)=\underline X^\nu\}
\end{align*}
be the supremum of $X$ and its time, and the infimum of $X^\nu$ and its time, we define two post extremal processes:
\begin{align}\label{eq:X_up_down}
&X^\downarrow(t)=X(\overline T+t)-\overline X,\quad t\in[0,\zeta-\overline T)\\ 
\nonumber&{X^\nu}^\uparrow(t)=X^\nu(\underline T^\nu+t)-\underline X^\nu,\quad t\in[0,\zeta^\nu-\underline T^\nu),
\end{align}
and assign $X^\downarrow(t)=\partial$ and ${X^\nu}^\uparrow(t)=\partial$ otherwise.
It is well-known, see~\cite{bertoin_splitting,chaumont_doney}, that $X^\downarrow$ and ${X^\nu}^\uparrow$ are time-homogeneous (sub-)Markov processes, 
such that when started away from zero their laws coincide with
the laws of $X$ and $X^\nu$ started at the corresponding levels and conditioned to stay negative and positive, respectively, explaining the notations and terminology.
For completeness and with almost no additional work, we provide this statement in a rigorous form in Lemma~\ref{lem:markov}.

Finally, we define another c\`adl\`ag process on the real line:
\begin{equation}\label{eq:proc_Y}
Y(t)=\1{t\geq 0}X^\downarrow(t)-\1{t<0}{X^\nu}^\uparrow((-t)-)
\end{equation}
for $t\in[-(\zeta^\nu-\underline T^\nu),\zeta-\overline T)$ and put $Y(t)=\partial$ otherwise, see the right hand side of Figure \ref{fig:ZY}.
Roughly speaking, we find the time of supremum of $Z(t)$ for $t\geq 0$ and for $t\leq 0$, delete the path in between these times and shift these supremum points into $(0,0)$.
Interestingly, the law of the processes $Z(t)$ can be easily recovered from the law of the process $Y(t)$ as shown in Theorem~\ref{thm:main}.

\subsection{An identity relating the laws}\label{sec:theorems}
Take a \levy process $X$ (not a CPP, neither a process with monotone paths) with the Laplace exponent $\psi(\theta)$, and a number $\nu\in\R$ such that $\psi(\nu)<\infty$.
Consider the processes $Z$ and $Y$ on the real line as they are defined in Section~\ref{sec:processes}. 
Recall that the left parts of the processes are killed with rate $p>0$ and the right parts with rate $q>0$.
Consider the set of c\`adl\`ag paths on the real line with values in $\R\cup\{\partial\}$ with Skorohod's topology,
and let $\mathcal B$ be the corresponding Borel $\sigma$-algebra. 
Now the following result relates the laws of $Z$ and $Y$.
\begin{theorem}\label{thm:main}
 For any $B\in \mathcal B,t\in\R$ and $x\geq 0$ it holds that 
\begin{align*}&\p((Z(T+s)-\overline Z)_{s\in\mathbb R}\in B, T\in\D t,\overline Z\in\D x)\\
&=Ce^{-\nu x+(\psi(\nu)+p-q)t}\p(Y\in B,-Y(-t)\in\D x)\D t,
\end{align*}
where 
\begin{align}\label{eq:CC}
 C=\frac{q\underline k(p+\psi(\nu),\nu)}{\underline k(q,0)}=\frac{p\overline k(q,0)}{\overline k(p+\psi(\nu),-\nu)}>0
\end{align}
and $\overline k(\alpha,\beta)$ and $\underline k(\alpha,\beta)$ are the bivariate Laplace exponents of the ascending and descending ladder processes respectively,
corresponding to~$X$ (without killing).
\end{theorem}
The bivariate Laplace exponents $\overline k(\alpha,\beta)$ and $\underline k(\alpha,\beta)$ are discussed in detail in Section~\ref{sec:k}, 
see also~\cite[Ch.\ 6.4]{kyprianou} and~\cite[Ch.\ VI.1]{bertoin}. We only note at this point that these exponents are unique up to a scaling constant 
(coming from the scaling of local times), which clearly can be arbitrary in the above result. The proof of Theorem~\ref{thm:main} is given in Section~\ref{sec:excursion}.

The following corollary considers non-defective processes, that is, $p=q=0$, when $\nu>0$ and $\psi(\nu)=0$. 
Recall from Section~\ref{sec:levy} that this implies that $X$ drifts to $-\infty$ and $X^\nu$ drifts to $+\infty$, that is, the supremum of $Z$ is finite.
\begin{corollary}\label{thm:cor}
Assume that $\psi(\nu)=0$ for some $\nu>0$.
 Then for any $B\in \mathcal B,t\in\R$ and $x\geq 0$ it holds that 
\begin{align}\label{eq:main_id}
  &\p((Z(T+s)-\overline Z)_{s\in\mathbb R}\in B, T\in\D t,\overline Z\in\D x)\\
  \notag&=C_0e^{-\nu x}\p(Y\in B,-Y(-t)\in\D x)\D t,
\end{align}
where 
\begin{equation}\label{eq:C0}
 C_0=\frac{\underline k(0,\nu)}{\underline k'(0,0)}=\frac{\overline k(0,0)\underline k(0,\nu)}{\overline k(1,0)\underline k(1,0)}>0,
\end{equation}
where the derivative $\underline k'(0,0)$ is with respect to the first argument.
\end{corollary}
The only non-trivial part of its proof concerns the identification of $C_0$, which is done in Section~\ref{sec:k}. Again, the scaling of $\underline k$ and $\overline k$ is arbitrary.

\begin{remark}\label{rem:CPP}
One would expect that similar results hold true for random walks, which then can be extended to CPPs as well. 
On the one side analysis of random walks is less technical, 
but on the other side one will have to distinguish between strict and weak ascending ladder times, left-most and right-most supremum times, as well as 
random walks conditioned to stay positive and conditioned to stay non-negative. 
\end{remark}

\subsection{Examples}
The bivariate Laplace exponents $\overline k(\alpha,\beta)$ and $\underline k(\alpha,\beta)$ can be given explicitly in a number of cases, some of which we consider below.
In all of these cases we compute the constants $C$ and $C_0$.
Recall that the process $Y$ is constructed from the conditioned processes $X^\downarrow$ and ${X^\nu}^\uparrow$, see~\eqref{eq:proc_Y}.
The conditioned \levy processes can be obtained in various possible ways, which we summarize in Section~\ref{sec:conditioned}.

\subsubsection{Spectrally-negative process}
Suppose $X$ is a spectrally-negative process, and so $\psi(\theta)$ exists for all $\theta\geq 0\wedge \nu$. 
Let $\Phi(q)$ be the right inverse of $\psi(\theta)$, i.e., $\Phi(q)$ is the right-most solution of $\psi(\theta)=q$.
According to~\cite[Sec.~6.5.2]{kyprianou} we may take
\begin{align}\label{eq:spec_neg_k}
 &\overline k(\a,\beta)=\Phi(\a)+\beta, &\underline k(\a,\beta)=\frac{\a-\psi(\beta)}{\Phi(\a)-\beta}
\end{align}
for $\alpha,\beta\geq 0$,
and so one easily obtains from either representation in~\eqref{eq:CC} that
\[C=\frac{p\Phi(q)}{\Phi(p+\psi(\nu))-\nu}\]
if $p+\psi(\nu)\geq 0$. The latter assumption may be dropped, because~\eqref{eq:spec_neg_k} can be analytically continued to $\a> \psi(\nu)$, see also~\eqref{eq:knu}.
Note also that the denominator in the expression of $C$ is always positive.

Under the condition of Corollary~\ref{thm:cor} we have $\Phi(0)=\nu$.
By continuity we obtain $\underline k(0,\nu)=\lim_{\beta\rightarrow \nu}\frac{-\psi(\beta)}{\nu-\beta}=\psi'(\nu)$ and 
then from either representation in~\eqref{eq:C0} we get
\begin{equation}\label{eq:C1}
C_0=\nu\psi'(\nu). 
\end{equation}

\subsubsection{Spectrally-positive process} Suppose $X$ is a spectrally-positive process then $\hat X=-X$ is spectrally-negative, and we can take
$\overline k(\a,\beta)=\hat{\underline k}(\a,\beta)$ and $\underline k(\a,\beta)=\hat{\overline k}(\a,\beta)$. So according to~\eqref{eq:CC} and \eqref{eq:spec_neg_k} written for the process~$\hat X$ we get
\[C=\frac{q\hat\Phi(p+\psi(\nu))+\nu}{\hat \Phi(q)},\]
where $\hat \Phi$ is the (right) inverse of $\hat\psi(\theta)=\psi(-\theta).$

Under the condition of Corollary~\ref{thm:cor} we have $\hat\Phi(0)=0$ and so we obtain from~\eqref{eq:C0} that
\begin{equation}\label{eq:C2}
 C_0=-\psi'(0)\nu=-\nu\e X(1).
\end{equation}

\subsubsection{Brownian motion}
Clearly, the above formulas should coincide if $X$ is both spectrally-negative and spectrally-positive process, that is, $X$ is a BM.
For simplicity we only consider the constant $C_0$, i.e., equations~\eqref{eq:C1} and~\eqref{eq:C2}.

In this case, $\psi(\theta)=\frac{1}{2}\sigma^2\theta^2+\mu\theta$ with $\sigma>0$ and $\mu<0$. Hence $\nu=-2\mu/\sigma^2$ and then
$\psi'(\nu)=\sigma^2\nu+\mu=-\mu=-\psi'(0)$, which shows that indeed the above formulas coincide and result in
\[C_0=-\mu\nu=\frac{2\mu^2}{\sigma^2}.\]
So choosing $\sigma=1$ and $\mu=-1/2$ we get $C=1/2$ confirming the result in~\cite{engelke}.

\subsubsection{More general examples}
There are examples of \levy processes with both positive and negative jumps with explicit bivariate exponents $\overline k(\alpha,\beta)$ and $\underline k(\alpha,\beta)$.
A rather general process of this type is given by an independent sum of an arbitrary spectrally-negative \levy process and a CPP with positive jumps characterized by a rational transform, see~\cite{lewis_mordecki_rational}
and~\cite{asm_avr_pist_PH} for the particular case of positive jumps having so-called phase type distributions. 
Similarly, one can treat a process with arbitrary positive jumps and finite intensity negative jumps characterized by a rational transform. 
Here we only mention that the resulting expressions are in terms of roots of certain equations.

In general the bivariate Laplace exponents $\overline k(\alpha,\beta)$ and $\underline k(\alpha,\beta)$, and hence the constants $C$ and $C_0$,
can be computed (at least theoretically) using a Spitzer-type identity, see e.g.~\cite[Thm.\ 6.16]{kyprianou}. 
This would require triple integration, assuming that one inverts the transform to obtain the distribution of $X(t)$.

\subsection{Stationary particles systems and mixed moving maxima processes}\label{sec:particle_systems}

The results of Section \ref{sec:theorems} provide an alternative representation
of the process $Z$ on the real line. Here we apply Corollary~\ref{thm:cor} to provide a better understanding of particle systems driven by 
\levy processes.  We anticipate that Theorem~\ref{thm:main} will be useful to analyze a more general particle system, where particles can die and be born. 

Let $X$ be a \levy process whose Laplace exponent 
fulfills $\psi(\nu) = 0$ for a $\nu>0$. Suppose that $X^\nu$ and
$Z$ are defined as above. 
It is easily seen that under these assumptions
the measure $\pi(\D x) = e^{-\nu x} \D x$ is invariant for both processes $X$
and $-X^\nu$. Let further $\{U_i, i\in\N\}$ be a Poisson point process on $\R$
with intensity measure $\pi(\D x)$ and let $Z_i$, $i\in\N$, be independent copies of
the process $Z$. We consider the system 
\begin{align}
  \label{Psi_1}\Psi_1 = \{U_i + Z_i, i\in\N\}
\end{align}
of particles started at the $U_i$'s and moving along the trajectories
of $X$ for $t \geq 0$ and $-X^\nu$ for $t<0$, respectively.
Then $\Psi_1$ is a Poisson point process on the space $\mathcal D$ of c\`adl\`ag
functions on $\R$. It follows from the results of \cite{sto2008,eng2013a} that the system 
$\Psi_1$ is stationary (or translation invariant), in the sense that 
for any $u\in\R$, the shifted system $\{U_i + Z_i(\ \cdot + u) , i\in\N\}$
has the same distribution as $\Psi_1$. This kind of systems has been
analyzed in \cite{kab2010a} in the case that the particles move along
Gaussian trajectories.

In the definition of $\Psi_1$, the point $t=0$ is an exceptional 
point at which each single particle changes from the trajectory of $-X^\nu$
to the trajectory of $X$. Stationarity of $\Psi_1$ shows that, in fact, $t=0$ is not special.
Furthermore, we will show that the particle
system $\Psi_1$ can be equivalently represented by a system $\Psi_2$,
generated by scattering the starting time points of the
particles uniformly over the real line and letting them move
along the trajectories of processes distributed as $Y$ in \eqref{eq:proc_Y}.
This also provides an alternative proof that the particle system is stationary.

As mentioned in the introduction, the pointwise maximum in~\eqref{eq:maxima} 
of the particles in $\Psi_1$ is a stationary, max-stable process that
generalizes the Brown-Resnick process in \cite{bro1977}.
From both a theoretical and a practical point of view, an 
important question is whether such a process has a stochastic
representation as a mixed moving maxima process as defined in \eqref{MMM}.
It implies that the process is mixing (cf., \cite{sto2008,dom2012a})
and can be efficiently simulated if the law of the spectral processes
$\p_F$ is known (cf.\ Section \ref{sec:conditioned} for details).
The equivalent representation of $\Psi_1$ in terms of the 
conditioned process $Y$ in the theorem below directly yields
a mixed moving maxima representation of $\eta$. We thus give an
affirmative answer to the open question of \cite{sto2008} on the 
existence of such a representation and, moreover, we provide
the law of the spectral processes.

\begin{theorem}\label{thm2}
  Let $\{(Y_i,T_i,V_i), i\in\N\}$ be a 
  PPP on $\mathcal D\times \R\times \R$ with intensity measure
  $C_0 \, \p_Y(\D \omega) \, \D t \, e^{-\nu v}\D v$, where $\p_Y$ is the law
  of the process in~\eqref{eq:proc_Y} and $C_0 > 0$ is given by
  \eqref{eq:C0}. 
  Then, $\Psi_1$ has the same distribution as the Poisson point process
  \begin{align*}
    \Psi_2 = \{V_i + Y_i(\cdot  - T_i), i\in\N\},
  \end{align*}
  on $\mathcal D$. Furthermore, the
  process $\eta$ in \eqref{eq:maxima} possesses
  the mixed moving maxima representation
  \begin{align}\label{MMM_result}
    \eta(t) \stackrel{d}{=} \max_{i\in\N} V_i + Y_i(t - T_i), \quad t\in\R.
  \end{align}
\end{theorem}

\begin{remark}
  The constant $C_0$ in the above theorem has an alternative
  representation
  \begin{align*}
    C_0^{-1} = \e \left[\int_\R \exp(\nu Y(t)) \D t \right].
  \end{align*}
This follows either directly from~\eqref{eq:main_id} or from a computation of $-\log \p(\eta(0) \leq x)$
using void probabilities of the PPP $\Psi_1$ on the one side and the PPP $\Psi_2$ on the other:
\[\nu^{-1} e^{-\nu x}=-\log \p(\eta(0) \leq x)=\nu^{-1} e^{-\nu x} C_0 \e \left[\int_\R \exp(\nu Y(t)) \D t\right].\]
\end{remark}

\begin{proof}
We first introduce some notation. For two measurable spaces
$(S_1,\mathcal S_1)$ and $(S_2,\mathcal S_2)$, a measurable function $m: S_1 \to S_2$
and a measure $\kappa$ on $S_1$, denote by $m_*\kappa$ the pushforward
measure of $\kappa$ under $m$, i.e., $m_*\kappa(E) = \kappa(m^{-1}(E))$, for 
all $E\in\mathcal S_2$. Further, let $\mathcal D^*$ be the Borel
subset of $\mathcal D$ of functions that drift to $-\infty$, that is,
$\mathcal D^* = \{ \omega \in \mathcal D : \lim_{|t| \to \infty} \omega(t) = -\infty\}$,
and note that $\p (Z\in \mathcal D^*) = 1$. For $\omega \in \mathcal D^*$
let 
\begin{align*}
  \overline \omega = \sup_{t\in\R} \omega(t), \qquad g_\omega = \inf\{ t\in\R: \omega(t)\vee \omega(t-) = \overline \omega\}.
\end{align*}
Let $\Gamma$ be the Poisson point process
$\{(U_i,Z_i), i\in\N\}$ on $\R \times \mathcal D^*$ with intensity measure
$\gamma(\D u\ \D\omega) =  e^{-\nu u} \D u \, \p_Z(\D\omega)$. We define the mapping $f$ by
\begin{align*}
  f: \R \times \mathcal D^* \to \mathcal D^* \times \R \times \R, 
  \quad (x,\omega) \mapsto \left( \omega( g_\omega +  \cdot \  ) - \overline\omega,\  g_\omega,\    x + \overline\omega \right).
\end{align*}
It is straightforward to check that $f$ is measurable. Moreover, it induces
a Poisson point process $f_*\Gamma = \{f(U_i,Z_i), i\in\N\}$ on 
$\mathcal D^* \times \R \times \R$ which has intensity measure 
$f_*\gamma$ by a general mapping theorem (cf., \cite{kin1993}). In fact, for
Borel sets $B\subset \mathcal D^*$, $I, E \subset \R$, we compute
\begin{align}
  \label{eq:int} &f_*\gamma( B\times I \times E) = \int_{f^{-1}(B\times I \times E)} \gamma(\D u\ \D\omega)\\
  \notag &= \int_{u\in\R} e^{-\nu u} \int_{t\in I} \int_{y\in E} \p\left( (Z(g_Z + s)-\overline Z)_{s\in\mathbb R}\in B, T\in \D t, u + \overline Z\in\D y\right) \D u\\
  \notag & =C_0 \int_\R e^{-\nu u} \int_I \int_E e^{-\nu (y-u)}\p\left(Y\in B,u-Y(-t)\in\D y\right)\D t \ \D u\\
  \notag& =C_0 \int_I  \int_\R\int_E e^{-\nu y} \p\left(Y\in B,u-Y(-t)\in\D y\right)\D u \ \D t
\end{align}
where the second last equation is a direct consequence of the
identity \eqref{eq:main_id}. For fixed $B\subset \mathcal D^*$, define the measure $\rho_B$ by
\begin{align*}
  \rho_B(D) = \int_\R\int_D \p\left(Y\in B,u-Y(-t)\in \D y \right) \D u,
\end{align*}
for all Borel sets $D\subset \R$, and note that 
\begin{align*}
  \rho_B(D) &= \int_\R \int_\R \1{ u\in D - y}  \D u \p\left(Y\in B,-Y(-t)\in \D y \right) = \p\left(Y\in B\right) \int_D du.
\end{align*}
Thus, $\rho_B$ is a multiple of Lebesgue measure and we obtain
together with~\eqref{eq:int}
\begin{align}
  \notag f_*\gamma( B\times I \times E) = C_0 \ \p\left(Y\in B\right)  
  \int_I \D t  \int_E e^{-\nu y} \D y.
\end{align}
In other words, the intensity measure of $f_* \Gamma$ factorizes
and equals the intensity measure of $\{(Y_i,T_i,V_i),i\in\N\}$.
Finally, let $h$ be the measurable mapping
\begin{align*}
  h: \mathcal D^* \times \R \times \R \to \mathcal D^*, 
  \quad (\omega, t, y) \mapsto y + \omega( \ \cdot - t),
\end{align*}
so that $h(f(x,\omega))=x+\omega(\cdot)$.
The induced PPP $h_*(f_* \Gamma) = \{h(f(U_i,Z_i)), i\in\N\}$ is thus nothing else 
than $\Psi_1$. Furthermore, it has the same intensity measure as $\Psi_2$ according to the construction of~$\Psi_2$. Taking pointwise maxima
within the two point processes yields the mixed moving maxima
representation~\eqref{MMM_result}.
\end{proof}


\section{Proofs}\label{sec:proofs}
Throughout this section we write $X_t$ instead of $X(t)$ and similarly for other processes which leads to somewhat cleaner expressions.
\subsection{Bivariate Laplace exponents}\label{sec:k}
Consider a (non-defective) \levy process $X$ as in Section~\ref{sec:levy}.
Define the running supremum and infimum processes: 
\begin{align*}
&\overline X_t=\sup_{s\in [0,t]} X_s, &\underline X_t=\inf_{s\in [0,t]} X_s,
\end{align*}
as well as all time supremum and infimum: $\oX=\oX_{\infty},\uX=\uX_{\infty}$.
Let $\overline L$ be the local time of the strong Markov process $\oX_t-X_t$ at $0$ and let $\overline n$ be the measure of its excursions away from~0, see e.g.~\cite[Ch.~4]{bertoin}.
Recall that $\overline L$ is defined in a unique way up to a scaling constant. Let also 
\begin{equation}\label{eq:k_def}
 \overline k(\a,\beta)=-\log \e (e^{-\a \overline L_1^{-1}-\beta \overline H_1};\overline L_1^{-1}<\infty)
\end{equation}
be the Laplace exponent of a bivariate ascending ladder process $(\overline L^{-1},\overline H),$ where $\overline L^{-1}_{t}=\inf\{s:\overline L_s>t\}$ and $\overline H_t=X_{\overline L_t^{-1}}$.
We also write $\underline L, \underline n$ and $\underline k(\a,\beta)$ for the analogous objects constructed from $-X$, i.e., we consider the strong Markov process $X_t-\uX_t$ 
(note also that $\underline H_t=-X_{\underline L_t^{-1}}\geq 0$ is a non-decreasing process).

Following~\cite{chaumont_supremum} we assume in the rest of this work that the local times are normalized so that 
\begin{align}\label{eq:scaling}
 \overline k(1,0)=\underline k(1,0)=1,
\end{align}
which implies, see e.g.~\cite{chaumont_supremum}, that
\begin{equation}\label{eq:kk}
\overline k(p,0)\underline k(p,0)=p,\forall p\geq 0.
\end{equation}
Let also $\overline d\geq 0$ and $\underline d\geq 0$ be the linear drifts of the subordinators $\overline L_t^{-1}$ and $\underline L_t^{-1}$.
We are ready to give a proof of Corollary~\ref{thm:cor}.
\begin{proof}[Proof of Corollary~\ref{thm:cor}]
We only need to compute $C_0$ as the limit of $C$ in~\eqref{eq:CC} as $p,q\downarrow 0$.
Note that $\{\overline L_1^{-1}<\infty\}=\{\overline L_\infty>1\}$, which can not happen a.s.\ when \mbox{$\e X(1)<0$}. Hence from~\eqref{eq:k_def} we find that
$\overline k(0,0)>0$ and similarly we conclude that $\underline k(0,0)=0$.
Using~\eqref{eq:kk} we write $C=\overline k(q,0)\underline k(p+\psi(\nu),\nu)$ which results in $C_0=\overline k(0,0)\underline k(0,\nu)$.
For arbitrary scaled $\underline k,\overline k$ we first scale them so that~\eqref{eq:scaling} holds, which results in the second representation of $C_0$ in~\eqref{eq:C0}.
Now the first representation of $C_0$ in~\eqref{eq:C0} is obvious.
\end{proof}

We will require the following expressions for the Laplace exponents $\underline k^\nu,\overline k^\nu$ of the ladder processes corresponding to $X^\nu$,
see also~\cite{baurdoux} and~\cite[Ch.\ 7.2]{kyprianou}.
\begin{lemma}
 For $\a,\beta\geq 0$ it holds that
\begin{align}\label{eq:knu}
 &\underline k^\nu(\alpha,\beta)=\underline k(\alpha+\psi(\nu),\beta+\nu), &\overline k^\nu(\alpha,\beta)=\overline k(\alpha+\psi(\nu),\beta-\nu).
\end{align}
\end{lemma}
\begin{proof}
The first equation follows immediately from the definition of $\underline k(\a,\beta)$ given by~\eqref{eq:k_def}. That is,  
\begin{equation*}
 \underline k^\nu(\alpha,\beta)=-\log \e (e^{\nu X_{\underline L_1^{-1}}-\psi(\nu) \underline L_1^{-1}}e^{-\a \underline L_1^{-1}+\beta X_{\underline L_1^{-1}}};\underline L_1^{-1}<\infty)=\underline k(\alpha+\psi(\nu),\beta+\nu),
\end{equation*}
and similarly for the second equation, where we have used the fact that the inverse local times are stopping times, see~\cite[Lem.\ 6.9]{kyprianou}.
\end{proof}

Moreover, we will need the following representation of the Wiener-Hopf factors
\begin{align}\label{eq:WH}
&\e e^{-\psi(\nu)\overline T_{e_p}+\nu\overline X_{e_p}}=\frac{\overline k(p,0)}{\overline k^\nu(p,0)},
&\e e^{-\psi(\nu)\underline T_{e_p}+\nu\underline X_{e_p}}=\frac{\underline k(p,0)}{\underline k^\nu(p,0)},
\end{align} 
where $\overline T_{e_p}$ and $\underline T_{e_p}$ are the time of supremum and the time of infimum respectively, see~\cite[Thm.\ 6.16]{kyprianou} and~\eqref{eq:knu}.
This requires an additional commentary, because strictly speaking the first identity holds for $\psi(\nu)\geq 0,\nu\leq 0$, and the second for $\psi(\nu)\geq 0,\nu\geq 0$.
Nevertheless these identities can be continued analytically to include arbitrary $\nu$ and $\psi(\nu)$ 
if we can show that the left sides are finite.
For this write
\begin{align*}
&1=\e e^{-\psi(\nu)e_p+\nu X_{e_p}}\geq \e (e^{-\psi(\nu)e_p+\nu X_{e_p}};\overline T_{e_p}< 1,\overline X_{e_p}<1)\\
&=\e e^{-\psi(\nu)(e_p-\overline T_{e_p})+\nu(X_{e_p}-\overline X_{e_p})}\e(e^{-\psi(\nu)\overline T_{e_p}+\nu \overline X_{e_p}};\overline T_{e_p}< 1,\overline X_{e_p}<1)
\end{align*}
and recall that $e_p-\overline T_{e_p}$ has the law of $\underline T_{e_p}$, and $X_{e_p}-\overline X_{e_p}$ the law of $\underline X_{e_p}$, see~\cite[Thm.\ 6.16]{kyprianou}.
This shows that $\e e^{-\psi(\nu)\underline T_{e_p}+\nu\underline X_{e_p}}<\infty$ and the other factor can be handled in a similar way.
Now we also see that 
\[\frac{\overline k(p,0)}{\overline k^\nu(p,0)}\frac{\underline k(p,0)}{\underline k^\nu(p,0)}=\e e^{-\psi(\nu)e_p+\nu X_{e_p}}=1\]
yielding
\begin{equation}\label{eq:kk_nu}
\overline k^\nu(p,0)\underline k^\nu(p,0)=p,\forall p> 0,
\end{equation}
in view of~\eqref{eq:kk}.

Finally, the following technical lemma is needed to claim that the time of supremum of $Z$ is a.s. not $0$. 
We say that $0$ is (ir)regular upwards if $0$ is (ir)regular for $(0,\infty)$.
\begin{lemma}\label{lem:no_point_mass}
 The point $0$ is irregular upwards for $X$ if and only if $0$ is irregular upwards for $X^\nu$.
\end{lemma}
\begin{proof}
 According to~\cite[Thm.~6.5]{kyprianou} and the observations in Section~\ref{sec:levy} we only need to consider the case when $X$ has paths of bounded variation and $\hat a=\hat a^\nu=0$. 
Since truncation of the \levy measure does not affect regularity issues we need to show that 
\[\int_{(0,1)}\frac{x\Pi(\D x)}{\int_0^x \Pi(-1,-y)\D y}<\infty\text{ iff }\int_{(0,1)}\frac{x\Pi^\nu(\D x)}{\int_0^x \Pi^\nu(-1,-y)\D y}<\infty,\]
which is obvious from the relation between $\Pi$ and $\Pi^\nu$.
\end{proof}

\subsection{Excursion theory and splitting}\label{sec:excursion}
In this section we adopt a very convenient notation of~\cite{duquesne,chaumont_supremum} and rely on Thm.~5 in~\cite{chaumont_supremum}.
We let $\Omega=\mathcal D$ be the space of c\`adl\`ag paths $\omega:[0,\infty)\rightarrow\mathbb R$ with lifetime $\zeta(\omega)=\inf\{t\geq 0: \omega_t=\omega_s,\forall s\geq t\}$.
The space $\Omega$ is equipped with the Skorohod's topology, and the usual completed filtration $(\mathcal F_t)_{t\geq 0}$ is generated by the coordinate process $X_t(\omega)=\omega(t)$.
We denote by $\p_q$ and $\p_q^\nu$ the laws of \levy processes $X$ and $X^\nu$ killed at rate $q>0$.
Similarly, $\p^\downarrow_q$ and $\p_q^\uparrow$ denote the laws of $X^\downarrow$ and $X^\uparrow$, which are the post-supremum and post-infimum processes of the killed $X$, see~\eqref{eq:X_up_down}.
Furthermore, ${\p^\nu_q}^\uparrow$ is used with the obvious meaning.
Note that in this setup instead of assigning $\partial$ at the killing time we keep the process constant. 
This setup will be sufficient to prove Theorem~\ref{thm:main}.

Let $\om^0$ be a path identically equal to 0, and define three operators on $\mathcal D$:
\begin{align*}
 \theta_t(\om)=(\om_t\vee \om_{t-}-\om_{t+u})_{u\geq 0},\\
k_t(\om)=(\om_u\1{u<t}+\om_t\1{u\geq t})_{u\geq 0},\\
r_t(\om)=(\om_t\vee\om_{t-}-\om_{(t-u)-}\1{u<t})_{u\geq 0},
\end{align*}
see Figure \ref{fig:operators}, as well as the usual shift operator: $s_t(\om)=(\om_{t+u}-\om_t)_{u\geq 0}$.
\begin{figure}[h]
\caption{Operators $\theta_s,k_s,r_s$ acting on $\om$.}\label{fig:operators} 
\begin{center}
\includegraphics{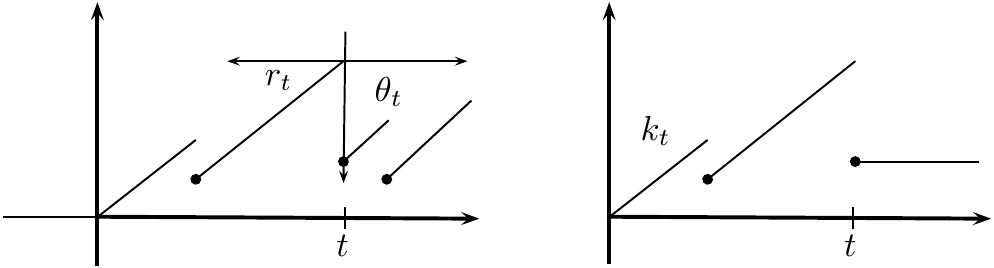}
\end{center}
\end{figure}

In the following we let $F,K$ be two bounded Borel functionals on $\mathcal D$ and put $g=\inf\{s\geq 0:X_s\vee X_{s-}=\oX\}$.
First, we note that 
\begin{align}\label{eq:updown}
 &\e_p^\uparrow (F)=\e_p (F\circ r_{g}), &\e_p^\downarrow (\tilde K)=\e_p (K\circ \theta_{g}) 
\end{align}
with $\tilde K(w)=K(-w)$, where the second follows directly from the definition of $X^\downarrow$, 
and the first from the definition of $X^\uparrow$ for a time-reversed process $X\circ r_{e_p}$,
which has the same law as the process $X\circ k_{e_p}$, see~\cite[Lem.\ II.2]{bertoin}. 

The following result is well-known, see e.g.~\cite[Lem.\ VI.6]{bertoin} and note that if there is a jump up at $g$ then it is necessarily the case (i) of this Lemma,
and if there is a jump down at $g$ then it is the case~(ii); otherwise there is no difference.
\begin{theorem}\label{thm:splitting}
For $p>0$ it holds that
\begin{align*}
\e_p (F\circ r_{g}\cdot K\circ \theta_{g})=\e_p (F\circ r_{g})\e_p (K\circ \theta_{g}),
\end{align*}
that is, the pre- and post-supremum processes are independent.
\end{theorem}
The following result expresses the law of pre- and post-supremum processes via excursion measures, see~\cite[Thm.\ 5]{chaumont_supremum}.
One can either extract the following identity directly from the proof of~\cite[Thm.\ 5]{chaumont_supremum}, 
or integrate the result of~\cite[Thm.\ 5]{chaumont_supremum} multiplied by $pe^{-p}$ and change the order of integration.
\begin{theorem}[Chaumont]\label{thm:chaumont}
For $p>0$ it holds that
\begin{align*}
&p\e_p (F\circ r_{g}\cdot K\circ \theta_{g})\\
  &=\left(\underline n(F\circ k_{e_p},e_p<\zeta)+p\underline dF(\omega^0)\right)\left(\overline n(K\circ k_{e_p},e_p<\zeta)+p\overline dK(\omega^0)\right).
\end{align*}
\end{theorem}
It is noted that $F(\omega^0)$ and $K(\omega^0)$ correspond to the events $\{g=0\}$ and $\{g=e_p\}$ respectively.
Furthermore, according to~\cite{chaumont_supremum} at least one of $\underline d$ and $\overline d$ is $0$ and
\begin{align}
&\underline k(p,0)=\underline n(e_p<\zeta)+p\underline d,
&\overline k(p,0)=\overline n(e_p<\zeta)+p\overline d. 
\end{align}
So picking $K=1$ and using~\eqref{eq:kk} we obtain
\begin{equation}\label{eq:ch1}
 \e_p (F\circ r_{g})=\left(\underline n(F\circ k_{e_p},e_p<\zeta)+p\underline dF(\omega^0)\right)/\underline k(p,0)
\end{equation}
and similarly for the other term $\e_p (K\circ \theta_{g})$,
which combined with Theorem~\ref{thm:chaumont} and~\eqref{eq:kk} proves Theorem~\ref{thm:splitting}.
Note also that~\eqref{eq:ch1} equals to $\e_p^\uparrow (F)$ according to~\eqref{eq:updown} and hence it specifies the law of the conditioned process in terms of the excursion measure.

Let us show that $X^\uparrow$ is a time-homogeneous Markov process, such that when started in $x>0$ its law coincides with the law of $X$ started in $x$ and conditioned to stay above $0$,
see also~\cite{bertoin_splitting, chaumont_doney}.
\begin{lemma}\label{lem:markov}
For $x>0$ it holds that
 \begin{align}\label{eq:markov_lem}
  &\e^\uparrow_p(F\circ k_t\cdot K\circ s_t,X_t\in\D x,t<\zeta)=\e^\uparrow_p(F\circ k_t,X_t\in\D x,t<\zeta)\e_p(K|\underline X>-x),
 \end{align}
which furthermore can be expressed as
\[\underline n(F\circ k_t,X_t\in\D x,t<e_p,t<\zeta)\e_p(K,\underline X>-x)/\underline k(p,0).\]
\end{lemma}
\begin{proof}
 According to~\eqref{eq:updown} and \eqref{eq:ch1} the left hand side (lhs) of~\eqref{eq:markov_lem} is given by
\[\underline n((F\circ k_t\cdot K\circ s_t,X_t\in\D x,t<\zeta)\circ k_{e_p},e_p<\zeta)/\underline k(p,0),\]
where the term containing $\omega^0$ results in 0, because $x>0$. Next, the right hand side (rhs) immediately reduces to
\[\underline n(F\circ k_t\cdot K\circ k_{e_p-t}\circ s_t,X_t\in\D x,t<e_p<\zeta)/\underline k(p,0).\]
Recall that $\underline n(\cdot|t<\zeta)$ is the law of the first excursion from the minimum of length larger than $t$, see~\cite[Ch.\ IV]{bertoin}.
The standard application of the strong Markov property of $X$ at the first time when its excursion from the minimum exceeds length $t$ yields the following identity:
\begin{align*}
&\underline n(F\circ k_t\cdot K\circ k_{e_p-t}\circ s_t,X_t\in\D x,t<e_p<\zeta|t<\zeta)\\
&=\underline n(F\circ k_t,X_t\in\D x,t<e_p|t<\zeta)\e_p(K,\underline X>-x),
\end{align*}
where $\underline X>-x$ in the second term signifies that the excursion length exceeds $e_p$. Here we also used the memoryless property of the exponential distribution.
This finally yields
\[\underline n(F\circ k_t,X_t\in\D x,t<e_p,t<\zeta)\e_p(K,\underline X>-x)/\underline k(p,0)\]
for the lhs of~\eqref{eq:markov_lem}. Plugging $K=1$ we obtain an expression for $\e^\uparrow_p(F\circ k_t,X_t\in\D x,t<\zeta)$, which then immediately leads to the result.
\end{proof}

The following identity for the pre-supremum process and $t>0$ will be important:
\begin{equation}\label{eq:pre_max}\e_p (F\circ r_{g},g\in\D t,\oX\in\D x)=pe^{-pt}\underline n(F\circ k_t,X_t\in \D x,t<\zeta)\D t/\underline k(p,0).\end{equation}
To see it observe that the lhs is 
\begin{align*}
 &\e_p((F\1{X_\zeta\in\D x,\zeta\in\D t})\circ r_{g})=\underline n((F\1{X_\zeta\in\D x,\zeta\in\D t})\circ k_{e_p},e_p<\zeta)/\underline k(p,0)\\
&=\underline n(F\circ k_{e_p},X_{e_p}\in \D x,e_p\in\D t,t<\zeta)/\underline k(p,0),
\end{align*}
where the second step follows from~\eqref{eq:ch1}, because $t>0$. The final expression is clearly the rhs of~\eqref{eq:pre_max}.
 
\begin{remark}\label{rem:scaling}
Recall that $\underline n/\underline k$ does not depend on the scaling of the local time process~\cite[Ch.\ IV]{bertoin}, 
and hence~\eqref{eq:pre_max} holds irrespective of the assumption~\eqref{eq:scaling}, and in particular it holds under measure change.
The same is true with respect to Lemma~\ref{lem:markov}.
\end{remark}

The following result, extending~(3.8) in~\cite{baurdoux}, expresses the excursion measures under measure change.
\begin{prop}\label{prop:tilting}
Let $\overline n^\nu$ and $\underline n^\nu$ be the excursion measures associated to $X$ under the measure $\p^\nu$.
Then for $t>0$ it holds that
\begin{align}
&\underline n^\nu(F\circ k_t,X_t\in \D x, t<\zeta)=e^{\nu x-\psi(\nu)t}\underline n(F\circ k_t,X_t\in \D x, t<\zeta),\label{eq:n_und}\\
&\overline n^\nu(F\circ k_t,X_t\in \D x, t<\zeta)=e^{-\nu x-\psi(\nu)t}\overline n(F\circ k_t,X_t\in \D x, t<\zeta).\label{eq:n_ov}
\end{align}
\end{prop}
\begin{proof}
By the definition of $\p^\nu$ we have
\begin{align*}
 &\e_p^\nu (F\circ r_{g},g\in\D t,\oX\in\D x)=\e_p(e^{\nu X_\zeta-\zeta\psi(\nu)} F\circ r_{g},g\in\D t,\oX\in\D x)\\
&=\e_p e^{\nu(X_\zeta-\oX)-\psi(\nu)(\zeta-g)}\e_p (e^{\nu \oX-\psi(\nu)g}F\circ r_g,g\in\D t,\oX\in\D x)\\
&=\frac{\underline k(p,0)}{\underline k^\nu(p,0)}e^{\nu x-\psi(\nu)t}\e_p (F\circ r_g,g\in\D t,\oX\in\D x),
\end{align*}
where we use Theorem~\ref{thm:splitting} (splitting at the supremum) in the second step, and~\eqref{eq:WH} in the third.
Combining this with \eqref{eq:pre_max} we obtain
\begin{equation*}
 \underline n^\nu(F\circ k_t,X_t\in \D x,t<\zeta)\D t=e^{\nu x-\psi(\nu)t}\underline n(F\circ k_t,X_t\in \D x,t<\zeta)\D t,
\end{equation*}
which proves~\eqref{eq:n_und} for Lebesgue almost all $t$. Extension to all $t$ can be done as in the proof of Thm.~5 of~\cite{chaumont_supremum}.

The equation~\eqref{eq:n_ov} follows immediately from~\eqref{eq:n_und} by considering the process $-X$ and changing measure according to $-\nu$.
Note that $(-X)^{-\nu}$ is just $-X^\nu$ and hence the lhs of~\eqref{eq:n_und} corresponds to the measure~$\overline n^\nu$. 
Finally, the first term on the rhs of~\eqref{eq:n_und} becomes $e^{-\nu x-\psi(-(-\nu))t}$, which completes the proof.
\end{proof}

We are now ready to give the proof of our main result.
\begin{proof}[Proof of Theorem~\ref{thm:main}]
  First suppose that $t>0$. According to splitting at the supremum of $X$, see Theorem~\ref{thm:splitting}, 
the post-supremum process $(Z_{T+s}-\overline Z)_{s\geq 0}$ given $T>0$ is independent of the rest (including the supremum and its time) and has the law of $(Y_s)_{s\geq 0}$.
 So we are only concerned with the pre-supremum process and the supremum with its time.
It is only required to show that
\begin{align}\label{eq:to_prove}
 &\e_q (F\circ r_g,g\in\D t,\oX\in\D x)\e_p^\nu(K,\uX>-x)\\
&=f(x,t){\e_p^\nu}^\uparrow (F\circ k_t \cdot K\circ s_t, X_t\in\D x,t<\zeta)\D t,\nonumber
\end{align}
where $f(x,t)=Ce^{-\nu x+(\psi(\nu)+p-q)t}$ and ${\p^\nu_p}^\uparrow$ is the law of the pre-supremum process of $Y$. 
Here we split the sample path of the pre-supremum process at time $t$ and apply functional $F$ to the first part and functional $K$ to the second.
See also the lhs of Figure~\ref{fig:proof}, where the additional axes show a convenient perspective on the sample path and its splitting.
\begin{figure}[h]
                \includegraphics{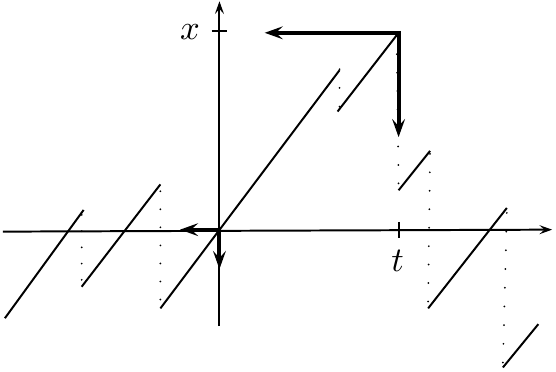}\hspace{0.3in}
                \includegraphics{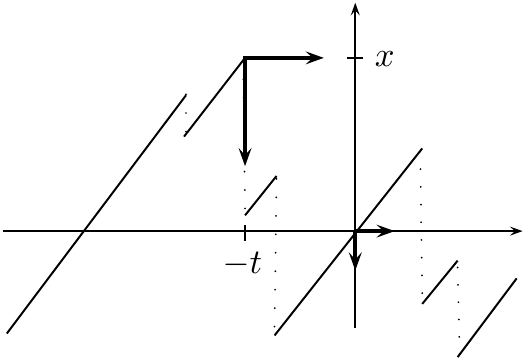} %
        \caption{Schematic sample paths of $Z$: the cases $T>0$ and $T<0$.}\label{fig:proof}
\end{figure}

According to~\eqref{eq:pre_max} the lhs of~\eqref{eq:to_prove} equals
\[qe^{-qt}\underline n(F\circ k_t,X_t\in \D x,t<\zeta)\e_p^\nu(K,\uX>-x)/\underline k(q,0)\D t.\]
According to Lemma~\ref{lem:markov} (the Markov property of $X^\uparrow$) the rhs of~\eqref{eq:to_prove} reduces to
\[f(x,t)e^{-pt}\underline n^\nu(F\circ k_t,X_t\in\D x,t<\zeta)\e^\nu_p(K,\underline X>-x)/\underline k^\nu(p,0)\D t,\]
see also Remark~\ref{rem:scaling}.
Using Proposition~\ref{prop:tilting} we see that these expressions indeed coincide when
$C=q\underline k^\nu(p,0)/\underline k(q,0)$, which is the left expression of $C$ in~\eqref{eq:CC} according to~\eqref{eq:knu}.


Next suppose that $t<0$. One can repeat the above arguments adjusting for the splitting at the infimum of $X^\nu$.
Instead of going this way and introducing additional notation, 
we simply consider the process $-X$ and change measure according to $-\nu$: $(-X)^{-\nu}$ is just $-X^\nu$, see also the second part of the proof of Proposition~\ref{prop:tilting}.
It is then required to prove for $t>0$ that
\begin{align*}
 &\tilde\e^{-\nu}_p (F\circ r_g,\oX\in\D x,g\in\D t)\tilde\e_q(K,\uX>-x)\\&=f(x,-t){\tilde\e_q}^\uparrow (F\circ k_t,K\circ s_t, X_t\in\D x,t<\zeta)\D t,
\end{align*}
where $X$ under $\tilde \e$ is the law of $-X$ under $\e$, see the rhs of Figure~\ref{fig:proof}. Similarly to the above derivation, the lhs equals
\[pe^{-pt}\overline n^\nu(F\circ k_t,X_t\in \D x,t<\zeta)\tilde\e_q(K,\uX>-x)/\underline k^\nu(p,0)\D t\] 
and the rhs equals to
\[f(x,-t)e^{-qt}\overline n(F\circ k_t,X_t\in\D x,t<\zeta)\tilde\e_q(K,\underline X>-x)/\overline k(q,0)\D t.\]
Again using Proposition~\ref{prop:tilting} we find that both sides are equal when
$C=p\overline k(q,0)/\overline k^\nu(p,0)$, which is the right expression of $C$ in~\eqref{eq:CC} according to~\eqref{eq:knu}.
Note that both expressions for $C$ coincide due to~\eqref{eq:kk_nu}.

According to Lemma~\ref{lem:no_point_mass} if $0$ is irregular upwards for $X$ then $0$ is regular downwards for $X^\nu$, because otherwise the later process would be a CPP, which is not the case.
Hence our construction~\eqref{eq:Z} does not allow $T$ to have a point mass at 0.
\end{proof}


\section{Conditioned processes}\label{sec:conditioned}

Simulation of the process $\eta$ in \eqref{eq:maxima} based on the 
particle system $\Psi_1$ in \eqref{Psi_1}
by simply sampling the $U_i$'s top down and adding to each of
them a realization $Z_i$ of the process $Z$ is problematic.
As $Z$ drifts to $-\infty$ almost surely, stationarity will only
be attained locally around $t = 0$ for finite sample sizes.
The equivalent mixed moving maxima representation based on $\Psi_2$
derived in Theorem \ref{thm2}
offers an appealing alternative sampling method (see also \cite{sch2002}):
Simulation of the points $(V_i,T_i)$ of the Poisson point process
with intensity $C_0\, \D t\, e^{-\nu v} \D v$ is straightforward. To each 
of these points, a realization $Y_i$ of the conditioned process
$Y$ has to be sampled. This is more subtle since the densities of 
this process are unknown in most cases. Below, we will therefore briefly list
several possibilities from the literature to obtain sample paths of $Y$, where 
each of the options is worth consideration.

The advantage of this procedure is that the maxima $T_i$ are 
scattered uniformly over the real line and thus global stationarity
is attained considerably faster than under simulation based on $\Psi_1$
(see Section 3 in \cite{engelke} for the case of Brownian motion).
For Brown-Resnick processes which correspond to Gaussian particle
systems, \cite{oes2012} used a similar method. There, the respective
constant $C_0$ is not known in closed form and its computation is
expensive. Thanks to formula \eqref{eq:C0}, in our case this is unnecessary.

As mentioned above, simulation of the conditioned process $Y$
is non-trivial. 
Note that $Y$ in Theorem \ref{thm2}
is composed from $-(-X)^\uparrow$ and $(X^\nu)^\uparrow$, where both $-X$ and $X^\nu$ drift to $+\infty$.
Hence for simplicity of notation in the following we assume that 
$X$ is a \levy process (not a CPP) drifting to $+\infty$, and discuss some alternative ways known in the literature to obtain the conditionally positive process $X^\uparrow$.

\begin{enumerate}
\item \emph{Post-infimum process}: as a first option we consider our definition of $X^\uparrow$ as the post-infimum process, see~\eqref{eq:X_up_down}.
 \item \emph{Conditioned process}: the process $X^\uparrow$ on $[t,\infty)$ given $X^\uparrow(t)=x$ equals in law to 
the process $X$ started in $x$ and conditioned to stay positive, see Lemma~\ref{lem:markov}.
Moreover, \cite[Thm.~2]{chaumont_doney} shows that $X^\uparrow$ can be approximated by the conditioned process started in~$x\downarrow 0$.
If $0$ is irregular upwards then this approximation holds for strictly positive times only, because in such a case $X^\uparrow(0)$ is not necessarily~0.
The distribution of the initial value of $X^\uparrow$ can be found in~\cite{chaumont_french} when $X$ has no negative jumps.
\item \emph{Excursions from the maximum}: \cite{doney_tanakas} extending~\cite{tanaka_rw} showed that $X^\uparrow$ can be obtained 
by time-reverting excursions of $X$ from the maximum and sticking them together. It is assumed here that $0$ is regular downwards.
\item \emph{Excursions from the minimum}: $X^\uparrow$ on $[0,t]$ can be simulated from the excursion measure $\underline n$ as specified by Lemma~\ref{lem:markov}. 
Another representation of a similar type is given in~\cite[Thm.\ 7]{tanaka_levy}
\item \emph{Path segments in~$[0,\infty)$}: \cite{bertoin_splitting} showed that $X^\uparrow$ can be obtained by sticking together path segments of $X$ in the positive half-line 
together with an appropriate correction according to the behavior of $X$ at~0.
\item \emph{Williams' representation}: we recall this representation for a process with no positive jumps as it is given in~\cite[Thm.\ 18 and Cor.\ 19]{bertoin},
and refer to~\cite[Thm. 4.1, Thm. 4.2]{duquesne} for the general case.
It holds that $X^\uparrow$ up to its last time below $x$, say $\sigma^\uparrow_x$, has the same law as $X$ time-reversed at its first passage over $x$.
Moreover, the evolution of $X^\uparrow$ after $\sigma^\uparrow_x$ is independent from the past and has the law of $X^\uparrow$.
\item \emph{Pitman's representation}: for a process with no positive jumps~\cite[Thm.\ 20]{bertoin} constructs $X^\uparrow$ from $X$ by subtracting twice the continuous part of the infimum of $X$ 
and by discarding the jumps of $X$ across its previous infimum. This results in a $3$-dimensional Bessel process in the case of a BM,~\cite{pitman}.
\end{enumerate}


In order to avoid some possible confusion with the term `conditioned to stay positive', 
it is noted that one can `condition' $X$ (started in $x>0$) to stay positive even when $X$ does not drift to $+\infty$, i.e., when $\p_x(\underline X>0)=0$. 
In this case there are various natural ways to do so, which lead to different laws.
For example,~\cite{hirano_conditioning} shows that the following two limits result in different laws:
\begin{align*}
 &\lim_{s\rightarrow\infty}\p_x(A|s<\tau_0^-)\label{eq:cond1},&\lim_{y\rightarrow\infty}\p_x(A|\tau_y^+<\tau_0^-),
\end{align*}
where $\tau_0^-$ and $\tau_y^+$ are the first passage times below 0 and above $y$, respectively, and $A$ is an event in $\mathcal F_t$ for some $t>0$. 
Finally, we note that these ambiguities disappear when $X$ drifts to $+\infty$ as required by Corollary~\ref{thm:cor}.

\section*{Acknowledgements}
Financial support by the Swiss National Science Foundation Projects 
200021-140633/1, 200021-140686 (first author), and
200020-143889 (second author) is gratefully acknowledged.

\bibliography{Levy_particle_systems.bib}
\bibliographystyle{abbrv}
\end{document}